\documentclass{amsart}

\usepackage{amsthm}
\usepackage{amssymb}
\usepackage{amsmath}
\usepackage{amsmath,amscd}
\usepackage{epsfig}
\usepackage{mathrsfs}
\usepackage{verbatim}

\newtheorem{theorem}{Theorem}
\newtheorem{corollary}{Corollary}
\newtheorem{lemma}{Lemma}
\newtheorem{proposition}{Proposition}

\newtheorem{definition}{Definition}

\newcommand{\er}[0]{\mathbf e_r}
\newcommand{\erp}[0]{\mathbf e_r'}
\newcommand{\ez}[0]{\mathbf e_z}

\title[Prescribed rate of curvature blowup]{ Minimal laminations  with prescribed convex curvature blowup}

\author{Stephen J. Kleene}
\address{Department of Mathematics, Brown University, Providence, RI 02906}

\begin{document}
\maketitle

\begin{abstract}
We construct minimal laminations with prescribed singularities on a line segment using perturbation techniques and PDE methods. In addition to the singular set, the rate of curvature blowup is also prescribable in our construction, and we show that all curvature blowup rates between quadratic and quartic arise. Our result generalizes an earlier result of the author and of Hoffman and White.
\end{abstract}

\section{introduction}
In this article, we use  PDE methods to construct embedded minimal disks  in a fixed ball containing a straight line segment, and with prescribed values  for $|A|^2$ along the line segment, subject to a very general constraint:

\begin{theorem} \label{MainTheorem}
Let $\mathcal{H}$ denote the helicoid and let $\lambda(\sigma) : [0 , 1] \rightarrow R^+$ denote a positive  ``scale function'' satisfying
\begin{align} \label{ScaleFunctionAssumptions}
\lambda <  C_0, \quad |\dot{\lambda}| \leq C_1 \lambda^\epsilon, \quad  |\ddot{\lambda}\lambda| \leq C_1 \lambda^\epsilon, \quad |\dddot{\lambda} \lambda^2|  \leq C_1  \lambda^\epsilon, 
\end{align}
for constants $\epsilon > 0$, $C_0 > 0$ and $C_1 > 0$, and where `` $\dot{}$ '' denotes derivation $\sigma$. Then given $\epsilon$ and $C_1$, there is $C_0$ and $R_0 > 0$ so that for all functions satisfying the conditions above, there is a surface $\mathcal{H}^*$ such that:
\begin{enumerate} 
\item $\mathcal{H}^*$ is embedded in the cylinder $\{x^2 + y^2 \leq R_0^2, 0 \leq z \leq 1 \}$ .
\item $\mathcal{H}^*$ is locally a scaled small perturbation of $\mathcal{H}$ in the following sense: There is $c > 0$ so that for $\sigma \in [0, 1]$, the surface
\begin{align} \notag
\tilde{\mathcal{H}}^*_\sigma : = \lambda^{-1} (\sigma)\left(\mathcal{H}^* - (0, 0, \sigma)\right) \cap B(0, c /\lambda(\sigma))
 \end{align}
 is a smooth perturbation of $\mathcal{H} \cap B(0, \epsilon /\lambda(p))$.
\end{enumerate}
\end{theorem}

Our Theorem recovers the   main results of \cite{HW}  and \cite{K1},   which were in the former case obtained using  geometric measure theory  and in the latter the Weierstrass Representation, and which we state in plain language below:
 
 \begin{corollary}[From \cite{HW} and \cite{K1}] \label{KHWTheorem}
Given a compact subset $K$  of a line segment in $R^3$, there exists a sequence of minimal disks, properly embedded in a ball  $B$ containing $K$, converging to a minimal lamination with singularities exactly on $K$.
 \end{corollary}
 The simplest case of Theorem \ref{KHWTheorem} arises when the subset $K$ coincides the with the line segment itself, in which case a sequence of such embedded minimal disks is realized by helicoids of decreasing scale. Several simpler, non-trivial cases precede Theorem \ref{KHWTheorem},  including the main theorems of \cite{Kh}, \cite{D} and the original result of Colding-Minicozzi in \cite{CM4}, all of which used the Weierstrass Representation.  
 
 One of the motivations for approaching the problem of constructing embedded minimal disks and their limit laminations with PDE techniques is their great flexibility and direct control on the geometry of the resulting surfaces they provide, at least relative to the Weierstrass Representation. For example, we are able to construct, for each compact set $K$ as in the statement of Theorem \ref{KHWTheorem}, smooth families of minimal laminations singularities exactly on $K$ and in particular we show that all super-quadratic curvature blowup rates arise among such laminations:
 
 \begin{corollary} \label{ConvexBlowupCorollary}
 For each $\epsilon > 0$, there exists an embedded minimal disk $\Sigma$, embedded in the half cylinder $\{ x^2 + y^2 \geq 1, z \geq 0\}$ in $R^3$, and so that 
 \begin{align} \notag
(1/C) h^{-2(1 + \epsilon)} \leq \sup_{z \geq h}| A_\Sigma|^2 \leq C h^{-2(1 + \epsilon)}
 \end{align}
 \end{corollary}
 
The case $\epsilon = 0$ was treated in \cite{BK1}, in which a minimal lamination of the enterior of a positive cone was exhibited. A theorem of Meeks, Perez and Ros (\cite{MPR}) gives that no such lamination laminates a ball containing the singularity.  The statement of Corollary \ref{ConvexBlowupCorollary} above can be strengthened slightly to say  that minimal surfaces described arise as limits of properly embedded minimal disks.  

	Unlike the Weierstrass representation techniques of \cite{Kh}, \cite{D},\cite{K1}, and \cite{CM4}, the PDE techniques employed here should have several straightforward generalizations. Firstly, there seems to be no major obstruction to proving and analogue of  Theorem \ref{MainTheorem} in which the straight line is replaced by an arbitrary smooth  curve.  Secondly,  our techniques seem to apply more or less directly to singly periodic  minimal surfaces other than the helicoid. In particular, it seems quite possible that applying our technique to Scherk towers would yield some extremely pathological examples of embedded minimal surfaces; namely, embedded in a punctured ball, with infinite topology and 
with no smooth extension to the whole ball. 	
	
	 The precise regularity of the curve required for such a construction to work is no greater the a $C^3$ requirement, although we have not systematically addressed this question in this article. Work of Colding-Minicozzi in \cite{CM1}, \cite{CM2}, \cite{CM3}  famously shows that $K$ must be contained in a Lipshitz curve, and more recently B. White in \cite{W1} has strengthened the regularity to $C^1$.  Meeks and Weber in \cite{MW} have shown that every $C^{1, 1}$ curve arises as the blowup set. J. Bernstein and G. Tinaglia have considered problems relating limit to  laminations in \cite{BT}.  Constructions related to the present have been undertaken in \cite{BK1} and \cite{BK2}.

\section{Preliminaries} \label{Preliminaries}
\subsection{ Basic notation and conventions}
Throughout this paper we make extensive use of cut-off functions, and we adopt the following notation:  Let $\psi_0:R \to [0,1]$ be a smooth function such that
\begin{enumerate}
 \item $\psi_0$ is non-decreasing
\item $\psi_0 \equiv 1$ on $[1,\infty)$ and $\psi_0 \equiv 0$ on $(-\infty, -1]$
\item $\psi_0-1/2$ is an odd function.
\end{enumerate}
For $a,b \in R$ with $a \neq b$, let $\psi[a,b]:R \to [0,1]$ be defined by $\psi[a,b]=\psi_0 \circ L_{a,b}$ where $L_{a,b}:R \to R$ is a linear function with $L(a)=-3, L(b)=3$.
Then $\psi[a,b]$ has the following properties:
\begin{enumerate}
 \item $\psi[a,b]$ is weakly monotone.
\item $\psi[a,b]=1$ on a neighborhood of $b$ and $\psi[a,b]=0$ on a neighborhood of $a$.
\item $\psi[a,b]+\psi[b,a]=1$ on $R$.
\end{enumerate}

\subsection{Norms and H{\"o}lder spaces}
\begin{definition}
Given a function $u \in C^{j, \alpha} (D)$, where $D \subset R^m$, the   $(j, \alpha)$ \emph{localized H\"older norm} is given by
\begin{align*}
\| u\|_{j, \alpha} (p) : = \| u : C^{j, \alpha} (D \cap B_1 (p))\|.
\end{align*}
We let $C^{j, \alpha}_{loc} (D)$ denote the space of functions for which $ \|- \|_{j, \alpha}$ is  pointwise finite. 
\end{definition}

\begin{definition}
Given  a positive function $f: D \rightarrow R$,  we let the space $C^{j, \alpha} (D, f)$ be the space of functions for which the \emph{weighted norm} $\| -: C^{j, \alpha} (D, f)\|$  is finite, where we take
\begin{align*}
\| u : C^{j, \alpha} (D, f) \| : = \sup_{p \in D} f(p)^{-1} \| u\|_{j, \alpha} (p)
\end{align*}
\end{definition}

\begin{definition}
Let $\mathcal{X}$ and $\mathcal{Y}$ be two Banach spaces with norms $\| - : \mathcal{X}\|$ and $\| - : \mathcal{Y}\|$, respectively. Then $\mathcal{X} \cap \mathcal{Y}$ is naturally a Banach space with norm $\| -:  \mathcal{X} \cap \mathcal{Y}\|$ given by
\begin{align}\notag
\| f : \mathcal{X} \cap \mathcal{Y} \| = \| f: \mathcal{X}\| + \| f : \mathcal{Y}\|. 
\end{align}
\end{definition}
Let $\mathcal{X}$ be a Banach space with norm $\| - : \mathcal{X}\|$ and suppose $S \subset \mathcal X$. For convenience, throughout the paper we will sometimes write $\| - : S\|$, where for any $f \in S$ we simply let
\[
\|f:S\|:= \|f: \mathcal X\|.
\]

\subsection{Estimating  homogeneous quantitites} \label{EstimatingHomogeneousQuantities}

In this section we record a formalized procedure for producing estimates for quantities defined on immersions that scale homogeneously. The quantities and results which we record have already appeared in \cite{BK2}, and for this reason we omit the proofs and instead refer the reader to \cite{BK2} for details. 

Let $E$ be the Euclidean space  $E : = E^{(1)} \times E^{(2)} =   R^{3 \times 2} \times R^{3 \times 4}$. We denote points of $E$ by $ \underline{\nabla}  = (\nabla, \nabla^2)$, where
\begin{align*}
\nabla = (\nabla_1, \nabla_2) \quad \nabla^2 = (\nabla^2_{1  1}, \nabla^2_{2  2}, \nabla^2_{1  2}, \nabla^2_{2  1}).
\end{align*}
We then consider functions $\Phi(\underline{\nabla})$ on $E$ with the property 
\begin{align*}
\Phi (c \underline{\nabla}) = c^d \Phi (\underline{\nabla})
\end{align*}
for real numbers $c$ and $d$. We call such a function a \emph{homogeneous function of degree $d$}.   A  homogeneous degree $d$ function has the property that its $j^{th}$ derivative $D^{(j)} \Phi $ is homogeneous degree $d - j$.

Notice $E$ is just a Euclidean space so for any $V \in E$, we make the identification $T_VE =E$.
We extend this for each $k \in \mathbb Z^+$ and observe that $D^{(k)}\Phi(V):E^k \to R$. For clarity we provide the following definition.
\begin{definition}
Let $k \in \mathbb Z^+$, $V, W_1, \dots W_k \in E$. Then
\begin{equation}\notag
\left.D^{(k)}\Phi\right|_V(W_1 \otimes \dots \otimes W_k):=D^{(k)}\Phi(V)(W_1 \otimes \dots \otimes W_k).
\end{equation}For brevity, we denote the $k$-th tensor product of $W$ with itself by
\[
\otimes^{(k)} W:= W\otimes \dots \otimes W.
\]
\end{definition}
\begin{definition}
 Given an immersion $\phi : D \subset R^2 \rightarrow R^3$, we set $\underline{\nabla} [\phi] :  = (\nabla \phi, \nabla^2 \phi)$. A \emph{homogeneous quantity of degree $d$} on $\phi$ is then a function of the form $ \Phi[\phi] : = \Phi (\underline{\nabla}[\phi])$ for some homogeneous function $\Phi$ on $E$.
 \end{definition}
 
  We refer to a map $\underline{\nabla} (s, z): D \subset R^2 \rightarrow E$ as an \emph{immersion} if the quantity
\begin{align} \label{fraka_def}
\mathfrak{a} (\underline{\nabla}) = \mathfrak{a} (\nabla) : = 2 \sqrt{\det{\nabla^T\nabla}}/|\nabla|^2 
\end{align}
is everywhere non-zero, and otherwise we refer to it simply as a \emph{vector field}. 
  
\begin{definition}\label{TRdef}
 Given an immersion $\underline{\nabla}$ and a vector field $\mathcal{E}$, we  set 
\begin{align} \label{TR1}
R_{\Phi, \mathcal{E}}^{(k)} (\underline{\nabla}) : =  \int_0^{1} \frac{(1 - \sigma)^k}{k!} \left. D \Phi^{(k + 1)} \right|_{\underline{\nabla}(\sigma)} (\otimes^{(k + 1)} \mathcal{E}) d \sigma
\end{align}
where $\otimes^{(k)} \mathcal{E}$ denotes the $k$-fold tensor product of $\mathcal{E}$ with itself and where $\underline{\nabla} (\sigma) : = \underline{\nabla} + \sigma \mathcal{E}$.

When $\underline{\nabla}$ and $\mathcal{E}$ are of the form $\underline{\nabla}  = \underline{\nabla} \phi$ and $\mathcal{E}  = \underline{\nabla} V$ we write
 \begin{align} \notag
 R_{\Phi , V}^{(k)} (\phi) : = R_{\Phi , \mathcal{E} }^{(k)} (\underline{\nabla}).
 \end{align}
 \end{definition}
 
 Note that $R_{\Phi , \mathcal{E}} (\underline{\nabla}) $ is simply the order $k$ Taylor remainder so that:
 \begin{proposition}
 We have 
 \begin{align}\label{TR2}
 \Phi (\underline{\nabla} + \mathcal{E}) - \Phi(\underline{\nabla}) - \left. D\Phi\right|_{\underline{\nabla}} (\mathcal{E}) - \ldots - \frac{1}{k !}\left. D^{(k )} \Phi \right|_{\underline{\nabla}} \left(\otimes^{(k)}\mathcal{E} \right) = R^{(k)}_{\Phi, \mathcal{E}} (\underline{\nabla})
 \end{align}
 \end{proposition}

\begin{proposition} \label{HQEstimates} There exists $\tilde \epsilon>0$ such that if $\underline{\nabla}: D \rightarrow E$ is an immersion and $\mathcal{E}: D \rightarrow E$ is a vector field satisfying
\begin{align} \notag
\|\mathcal{E}: C^{j, \alpha}(D,\mathfrak{a} (\nabla) |\nabla|)\| \leq C(j, \alpha)\tilde \epsilon , \quad \text{ and }\quad \ell_{j, \alpha} (\nabla) : =  \| \nabla: C^{j, \alpha}(D, |\nabla|)\|  < \infty,
\end{align}
then
\begin{align} \notag
\left\| R_{\Phi , \mathcal{E}}^{(k)} (\underline{\nabla}): C^{j, \alpha} (D, |\nabla|^d) \right\| \leq C(\Phi,  \ell_{j, \alpha}, \mathfrak{a},k )\left\| \mathcal{E} : C^{j, \alpha} (D, |\nabla| ) \right\|^{k+1}.
\end{align}
\end{proposition}

\section{Outline}
We fix a  positive function $\lambda(\sigma): (0, 1) \rightarrow R$  satisfying the conditions of Theorem \ref{MainTheorem} and set
\begin{align} 
z (\sigma) : = \int_0^{\sigma} \frac{d \sigma}{\lambda (\sigma)}.
\end{align}
Let  $\sigma (z)$ denote the inverse of $z (\sigma)$. Then for any object $\Phi(\sigma)$ depending on the parameter $\sigma$, we will throughout simply write
\begin{align} \notag
\Phi(z) : = \Phi \circ \sigma (z).
\end{align}
We will also use `` $\dot{}$ '' to denote derivation in  $\sigma$ and `` $'$ '' to denote the derivative in $z$. We observe that 
\begin{align} \notag
\Phi' (z) = \lambda (z) \dot{\Phi} (z).
\end{align}
The map
\begin{align} \notag
B(x,y,z) = \sigma (z) e_z + \lambda(z) x e_x + \lambda (z) x e_y 
\end{align}
is then easily seen to be a diffeomorphism of $R^3$ and we set
\begin{align} 
F^* (s, z): = B \circ F(s, z)
\end{align}
where $F(s, z)$ is the conformal parametrization of the helicoid given by
\begin{align}
F(s, z) : =  \cosh (s) \sin (z) e_x + \cosh (s) \cos(z) e_y + e_z.
\end{align} 

Our goal is then to find a graph over $G$ which gives an embedded minimal surface as in Theorem \ref{MainTheorem}. As a first step we need to estimate various geometric quantities on $G$ including the mean curvature, metric, and the stability operator

\begin{proposition}\label{FStarGeometry}
Let $g^*$, $\nu^*$, $A^*$, $|A^*|$, $\Delta^*$,  $\mathcal{L}$ and $H$  denote respectively the metric, unit normal,  second fundamental form, the length of the second fundamental form, Laplace-Beltrami operator , and  the stability operator and the mean curvature of $F^*$ . Then:
\begin{enumerate}
\item With $g^* = g_{s s} ds^2 + g_{zz} dz^2 + 2 g_{s z} ds dz$, we have
\begin{align} \notag
g_{s s} = \lambda^2 \cosh^2(s), \quad g_{z z} = \lambda^2 \left( \cosh^2(s) + \dot{\lambda}^2 \sinh^2(s)\right), \quad g_{s z} = \dot{\lambda}  \lambda^2 \sinh(s) \cosh(s)
\end{align}
\item We have that
\begin{align} \notag
 \nu^* =  - \cosh^{-1}(s)e_r(x)  + \tanh(s) e_z,
\end{align}
 where we have set $e_r(x) = \sin (x) e_x + \cos(x)e_y$. \\
\item  With $A^* = A_{s s} ds^2 + A_{zz} dz^2 + 2 A_{sz} ds dz $ we have
\begin{align} \notag
A_{s s} = 0, \quad A_{z z}  - \dot{\lambda} \lambda \tanh(s), \quad A_{s z}  = - \lambda.
\end{align}
\item It holds that 
\begin{align} \notag
\lambda^2 \cosh^{2} (s)\Delta^*  &  = (1 + \dot{\lambda}^2\tanh^2(s)) \partial_{s  s} + \partial_{z  z} - 2 \dot{\lambda} \tanh(s) \partial_{s  z}   \\ \notag
& \quad +  \left\{ 2 \dot{\lambda}^2 \tanh(s)\cosh^{ - 2}(s) - \ddot{\lambda} \lambda \tanh(s)  \right\}  \partial_s -   \dot{\lambda} \cosh^{-2} (s)  \partial_z. 
\end{align}
\item It holds that 
\begin{align}
 |A^*|^2 = \lambda^{-2} \cosh^{-4 }( s)   \left( 2 + 2\dot{\lambda}^2\tanh^2(s)\right). 
\end{align}
. \\

 \item Set 
 \begin{align} \notag
\tilde{ \mathcal{L}} = \Delta + 2 \cosh^{-2} (s)
 \end{align}
 where $\Delta = \partial^2_{s \,s} + \partial^2_{z \, z}$ is the flat laplacian on $R^2$. Then we can write
 \begin{align} \notag
\lambda^2 \cosh^{2} (s)  \mathcal{L}^*  =  \tilde{\mathcal{L}} + \mathcal{E},
 \end{align}
where $\mathcal{E}$ is the operator given explicitly by
\begin{align} \notag
\mathcal{E}(s, z) &  = 2 \dot{\lambda}^2 \sinh^2(s)  +  \dot{\lambda}^2\tanh^2(s)  \partial_{s  s} - 2 \dot{\lambda} \tanh(s) \partial_{s  z}   \\ \notag
& \quad +  \left\{ 2 \dot{\lambda}^2 \tanh(s)\cosh^{ - 2}(s) - \ddot{\lambda} \lambda \tanh(s)  \right\}  \partial_s -   \dot{\lambda} \cosh^{-2} (s)  \partial_z.
\end{align}

\item  $H^*  = - \dot{\lambda} \lambda \tanh(s) \cosh^{-2} (s)$.
\end{enumerate}
\end{proposition}

\begin{definition} \notag
Given a function $u: R^2 \rightarrow R$ we set
\begin{align}
F^*[u] (s, z): = F^* (s, z)+ \lambda(z) u(s, z)\nu^*(s, z)
\end{align}
\end{definition}

From this and Proposition \ref{FStarGeometry} we get the immediate corollary:
\begin{proposition}\label{FStarGraphMC}
Let $u: R^2 \rightarrow R$ be a locally class $C^{2, \alpha}$ function on on $R^2$ and assume that 
\begin{align}
\| u : C^{2, \alpha} (R^2, \cosh(s))\| \leq \epsilon.
\end{align}
Then for $\epsilon > 0$ sufficiently small, $F^*[u]$ is a class $C^{2, \alpha}$ immersion and it holds that 
\begin{align}
\tilde{H}^*[u]  : = \lambda(z) \cosh^{2} (s) H [F^*[ u]] = \dot{\lambda}  \tanh(s) + \tilde{\mathcal{L}} u + \tilde{R}^*[u]
\end{align}
where $R^*$ satisfies the estimate
\begin{align} \notag
\| R^*[u] \|_{0, \alpha} \leq C\left( \dot{\lambda} \| u\|_{2, \alpha} +  \| u\|_{2, \alpha}^2. \right)
\end{align}
\end{proposition}
Once Proposition \ref{FStarGraphMC} is established, it remains to study the linear problem for $\tilde{\mathcal{L}}$:
\begin{align}\label{NormalizedTotalLinearProblem}
\tilde{\mathcal{L}} u = \dot{\lambda} \tanh(s).
\end{align}
Most of the inhomogeneous term in (\ref{NormalizedTotalLinearProblem}) can be integrated directly: It is easily checked that $\tanh(s)$ is in the kernel of 
\begin{align}
\tilde{L} : = \partial^2_{s \, s} + 2 \cosh^{-2} (s).
\end{align}
 Variation of parameters then gives that  
 \begin{align}
 \tilde{L}^{-1} (f) : = \left( \int_0^s \tanh^{-2} (s')\int_0^{s'} \tanh(s'') f(s'') ds'' ds'  \right) \tanh(s)
 \end{align}
is an inverse for $\tilde{L}$  that preserves zero dirichlet condition at $s = 0$. Set:
\begin{align} \label{PeriodicSolution}
u_0(s, z) : = \tilde{L}^{-1} (\tanh(s))
\end{align}
Then we have directly:
\begin{proposition}
The following statements hold:
\begin{enumerate}
\item It holds that 
\begin{align} \notag
 \tilde{\mathcal{L}} \left(\dot{\lambda} u_0 \right) = \dot{\lambda} (z)\tanh(s) + \left( \dddot{\lambda} \lambda^2 + \ddot{\lambda}^2 \lambda  \right) z) u_0   : =  \dot{\lambda} (z)\tanh(s) + E,
 \end{align}
  where $E$ is determined implicitly. \\ 
\item The function $u_0$ satisfies the estimate:
\begin{align} \notag
 \|u \|_{j, \alpha}(s) \leq 1 + s^2.
\end{align}

\end{enumerate}
\end{proposition}
Thus, we are left with solving the linear problem for the remainder term $E$. Using the convexity assumption (\ref{ScaleFunctionAssumptions}) we then have directly that 
\begin{align} \label{RemainderTermGlobalEstimate}
\|E : C^{0 , \alpha}(R^2, \lambda^\epsilon(1+ s^2)) \| \leq C C_1.
\end{align}
Since the $s$ parameter is approximately the logarithmic radial distance of $\lambda^{-1}(z)F^*(z, s)$ from the $z$ axis, the domain of our graph must contain at least the set  $\{(s, z) : |s| \leq \log (c/\lambda (z)) \}$ if we wish to obtain minimal surfaces in tubes of fixed radius independent of $\lambda$. We in fact include much more and find solutions on the domain
\begin{align} \label{LambdaDef}
\Lambda : = \{ (s, z): |s| \leq \underline{\ell} (z)\}, \quad \underline{\ell} (z) : =  \lambda^{- \tau \epsilon}(z)
\end{align}
where $\tau$ is a small positive constant to be determined.  Then with $\epsilon_0 : = (1- 2 \tau) \epsilon$, the definition of $\Lambda$ and the estimate for $E$ in (\ref{RemainderTermGlobalEstimate}) give 
\begin{align} \label{ELambdaWeightedEstimate}
\|E : C^{0, \alpha} (\Lambda, \lambda^{\epsilon_0}) \| : = \gamma < \infty.
\end{align}
Our main invertibility statement is then: 
\begin{proposition} \label{MainInvertibilityStatement}
Let $E: \Lambda \rightarrow R$ be a locally class $C^{0, \alpha}$ function satisfying the estimate in (\ref{ELambdaWeightedEstimate}). Then  there is a function $u: \Lambda \rightarrow R$ such that 
\begin{enumerate}
\item It holds that 
\begin{align} \notag
\tilde{\mathcal{L}} u = E
\end{align}

\item Set $\epsilon_2 = (1-  20\tau) \epsilon_0$. Then $u$ satisfies the estimate:
\begin{align} \notag
\| u : C^{2, \alpha} (\Lambda, \lambda^{\epsilon_1})\| \leq  C \|E : C^{0, \alpha} (\Lambda, \lambda^{\epsilon_0}) \|
\end{align}
\end{enumerate}
\end{proposition}
The primary difficulty in proving Theorem \ref{MainInvertibilityStatement} has to do with interactions of regions of $\Lambda$ corresponding to different  scales. If one assumes that the ratio
\begin{align} \notag
 \lambda_{\text{min}}/\lambda_{\text{max}}
\end{align}
 is uniformly bounded below, then Proposition \ref{MainInvertibilityStatement} becomes significantly easier. We essential prove Proposition \ref{MainInvertibilityStatement} as a corollary to Proposition \ref{StripInvertibilityWithRBC} below, which considers inhomogeneous terms supported on  strips of a fixed height and satisfying a strong orthogonality condition.
 
 \begin{definition} \label{LambdaSubDomains}
 The domain $\Lambda(a, b)$ is given as follows:
 \begin{align} \notag
 \Lambda (a, b): = \{|z| \leq b,  |s| \leq a \}.
 \end{align}
 \end{definition}
 
 \begin{proposition} \label{StripInvertibilityWithRBC}
Let $E$ be a class $C^{0, \alpha}$ function supported on $\Lambda(\ell, 2 \pi)$ and satisfying
\begin{align} \notag
\int_0^\ell E (s , z) \tanh(s) ds = 0.
\end{align}
Then, given $N > 2 \pi$ there is function $u_N : \Lambda(\ell, N  ) \rightarrow R$ such that 
\begin{enumerate}
\item \label{StripSolEq}  It holds that
\begin{align} \notag
\tilde{\mathcal{L}} u_N : = E.
\end{align}
\item \label{SolutionRobinBoundaryCondition} For $|z| \leq N$, $u(s, z)$ satisfies the Robin boundary condition
\begin{align} \notag
\partial_su_N (\ell, z) \tanh(\ell) - u_N(\ell, z)/\cosh^2(\ell) = 0.
\end{align}

\item \label{SolutionNeumanBoundaryCondition} $u_N$ satisfies the Neuman boundary condition
\begin{align} \notag
\partial_z u_N(s, \pm N) = 0.
\end{align}
\end{enumerate} 
 \end{proposition}
 
 An immediate consequence of Proposition \ref{StripInvertibilityWithRBC}  is that the orthogonality condition on $E$ is inherited by the solution $u_N$. This implies a one dimensional  Poincare inequality for $u_N$ along lines that gives  uniform control on the solutions $u_N$ in $N$.
 
 \begin{proposition} \label{RBCSolutionsIntegralControl}
 There is a universal constant $C$ so that the function $u_N$ in the statement of Proposition \ref{StripInvertibilityWithRBC} satisfies the following estimate:
 \begin{align} \notag
 \| u_N : L^1 (\Lambda(\ell, N)) \|, \quad \| \partial_s u_N : L^1 (\Lambda(\ell, N)) \| \leq C \ell^2 \| E : C^{0, \alpha}(\Lambda(\ell, 2 \pi), 1)\|.
 \end{align}
 \end{proposition}
 
From this, it then immediately follows that the functions $u_N$ are uniformly bounded in $N$ and  converge smoothly as $N \rightarrow \infty$ to a limiting function $u_\infty$ solving the linear problem for $E$ on the domain $\Lambda(\ell) : = \Lambda(\ell, \infty)$. This and a maximum principle implies the weighted H{\"o}lder estimate for $u_\infty$:
 \begin{proposition} \label{RBCSolutionsWeightedControl}
There is a constant $C$ so that 
 \begin{align} \notag
\left\| u_\infty : C^{2, \alpha} \left(\Lambda(\ell), \frac{1}{1 + |z|} \right)  \right\| \leq C \ell^3   \| E : C^{0, \alpha}(\Lambda(\ell, 2 \pi), 1)\|.
\end{align}
 \end{proposition}
Proposition \ref{MainInvertibilityStatement} then follows by applying Proposition \ref{RBCSolutionsWeightedControl} to pieces of the inhomogeneous term that lie in strips of fixed with, summing the resulting solutions, and using the convexity of the scale function $\lambda$ to show that only regions of comparable scale interact. The existence of the minimal graph over $F^*$ defined on $\Lambda$ then follows almost immediately. We formulate its existence  in terms of a fixed point  of the  mapping $\Psi$ on the ball
\begin{align} \label{FixedPointSet}
\Xi : = \{ u \in C^{2, \alpha}_{\text{loc}}(\Lambda) : \| u : C^{2, \alpha} (\Lambda, \lambda^{\epsilon_1}) \| \leq \zeta \},
\end{align}
where $\Psi$ is given by
\begin{align} \label{FixedPointMap}
\Psi (u) : = u - \tilde{\mathcal{L}}^{-1} \tilde{H}^*{u},
\end{align}
  where $\zeta$ is a constant to be determined, and where $\tilde{\mathcal{L}}^{-1}$ denotes the inverse to $\tilde{\mathcal{L}}$ between the weighted H{\"o}lder spaces described in Proposition \ref{MainInvertibilityStatement}. 
 \begin{proposition} \label{ContractionMapping}
For $\lambda$ sufficiently small, the function $\Psi (u)$ is defined on the set $\Xi$ and acts as a contraction. 
 \end{proposition}
 
 \section{Proof of Proposition \ref{FStarGeometry} and Corollary \ref{FStarGraphMC}}
 To prove Proposition \ref{FStarGeometry} we first record the first and second derivatives of $F^*$.
 
 \begin{lemma}
We have
\begin{align} \label{F^*Derivatives}
& F^*_s  =  \lambda   \cosh(s) \er , \quad F^*_z =  \lambda'  \sinh(s)\er + \lambda  \sinh(s) \erp  +  \lambda  \ez \\ \notag
& F^*_{s \,  s} =    \lambda \sinh(s) \er , \quad F^*_{s  \,  z} =  \lambda'  \cosh(s)\er +  \lambda \cosh(s) \erp, \\ \notag
&  F^*_{z \,   z} =  \lambda''  \sinh(s) \er  - \lambda  \sinh(s) \er + 2 \lambda' \sinh(s) \erp + \lambda'  \ez.
\end{align}
\end{lemma}

\subsection{The unit normal}

\begin{lemma} \label{UnitNormal}
The unit normal of ${F^*}$ is 
\begin{align}\label{normal}
\nu (s, z) = - \cosh^{-1}(s)\erp  + \tanh(s) \ez
\end{align}
\end{lemma}
\begin{proof}
This follows immediately as
\begin{align*}
F^*_s \wedge F^*_z & =  ( \lambda (z) \cosh(s) \er) \wedge( \lambda(z) \sinh(s)\er + e ^{\delta z} \sinh(s) \erp  +   \lambda (z) \ez ) \\ \notag
& =  \lambda^2 \sinh(s) \er \wedge \erp + \lambda^2 \cosh(s) \er \wedge \ez\\ \notag
&  = \lambda^2 \cosh(s) \sinh(s)\ez - \lambda^2 \cosh(s) \erp
\end{align*}
and thus
\begin{align*}
|F^*_s \wedge F^*_z|^2 = \lambda^4\cosh^4(s).
\end{align*}
\end{proof}

\subsection{The metric}

\begin{lemma} \label{Metric}
Let $g^* = g^* _{s  s} ds^2  + g^*_{z  z} dz^2 + 2 g^*_{s  z} ds \, d z$ be the metric of $g^*$. Then
\begin{align*}
&g^*_{s s} =  \lambda^2 \cosh^2(s), \quad g^*_{z  z} = \lambda^2 \left( \cosh^2(s) + \dot{\lambda}^2 \sinh^2(s)\right), \quad  g^*_{s z} = \dot{\lambda}  \lambda^2 \sinh(s) \cosh(s).
\end{align*}
\end{lemma}
\begin{proof}
This follows directly from (\ref{F^*Derivatives}).
\end{proof}

As a direct consequence, 
\begin{align} \label{determinant}
|g^*| : = \det g^* =  \lambda^4 \cosh^4(s)
\end{align}
and the components of the dual metric are
\begin{align} \label{dual_metric}
F^{*s  s} = \lambda^{-2}\cosh^{-2 }(s) (1 + \dot{\lambda}^2 \tanh^2(s)), \quad F^{*z z} = \lambda^{-2}\cosh^{-2}(s), \quad F^{*s  z} = - \dot{\lambda} \lambda^{ - 2}  \tanh(s) \cosh^{-2} (s).
\end{align}

\subsection{The second fundamental form}

\begin{lemma} \label{2FF}
Let $A :  = A_{s  s} ds^2 + A_{z  z} d z^2 + 2 A_{s  z} ds\,  d z$ be the second fundamental form, and let $|A|^2$ be its length. Then we have
\begin{align*}
A_{s  s} = 0, \quad A_{z z} = - \dot{\lambda} \lambda \tanh(s), \quad A_{s z} = - \lambda
\end{align*}
and 
\begin{align*}
|A|^2 = \lambda^{-2} \cosh^{-4 }( s)   \left( 2 + 2\dot{\lambda}^2\tanh^2(s)\right).
\end{align*}
\end{lemma}
\begin{proof}
We determine the components of the second fundamental form by using (\ref{F^*Derivatives}) and \eqref{normal}. To obtain the expression for the length of the second fundamental form, we write
\begin{align*}
|A|^2 &= 
\left(\begin{array}{ccc}
A_{s  z} A_{s  z} & A_{s  z} A_{z  s} & A_{s  z} A_{z  z} \\
& \\
A_{z  s} A_{s  z } &  A_{ z  s} A_{z  s} &A_{z s} A_{z z} \\
& \\
A_{z  z} A_{s  z} & A_{z z} A_{z  s} & A_{z  z} A_{z  z}
\end{array} \right) * 
 \left(\begin{array}{ccc}
g^{s  s} g^{z  z} & g^{s  z} g^{z  s} & g^{s  z} g^{z  z}\\
& \\
g^{z  s} g^{s  z}&  g^{z  z} g^{s  s} & g^{z  z} g^{s  z}\\
& \\
g^{z  s} g^{z  z} & g^{z  z} g^{z  s} & g^{z z} g^{z  z}
\end{array} \right) \\ \notag
\\
& = 
\lambda^2
\left(\begin{array}{ccc}
1 & 1 &  \dot{\lambda}  \tanh(s)\\ \notag
 \\
1 & 1 &  \dot{\lambda} \tanh(s)\\ \notag
\\
 \dot{\lambda} \tanh(s) &  \dot{\lambda} \tanh(s) & \dot{\lambda}^2 \tanh^2(s) 
\end{array} \right)
*
\frac{\lambda^{-4}}{\cosh^{ 4} (s) }
\left(\begin{array}{ccc}
1 + \dot{\lambda}^2\tanh^2(s) & \dot{\lambda}^2\tanh^2 (s) & - \dot{\lambda}\tanh(s) \\
\\
\dot{\lambda}^2\tanh^2 (s) & 1 + \dot{\lambda}^2\tanh^2(s) &  - \dot{\lambda}\tanh(s) \\
\\
-\dot{\lambda}\tanh(s) & - \dot{\lambda}\tanh(s) & 1
\end{array} \right) \\ \notag
\\
& = \cosh^{-4 }( s)  \lambda^{-2} \left( 2 + 2\dot{\lambda}^2\tanh^2(s)\right).
\end{align*}
\end{proof}

\begin{lemma} \label{MeanCurvature}
Let $H^*$ be the mean curvature of $F^*$. Then 
\begin{align*}
H^* (s, z) = - \dot{\lambda} \lambda \tanh(s) \cosh^{-2} (s).
\end{align*}
\end{lemma}

\subsection{The Laplace operator}
\begin{lemma} \label{LaplaceOperator}
Let $\Delta^*$ be the Laplace operator on $F^*$. Then 
\begin{align*}
\lambda^2 \cosh^{2} (s)\Delta_{g}  & = (1 + \dot{\lambda}^2\tanh^2(s)) \partial_{s  s} + \partial_{z  z} - 2 \dot{\lambda} \tanh(s) \partial_{s  z}  \\ \notag
&\quad +  \left\{ 2 \dot{\lambda}^2 \tanh(s)\cosh^{ - 2}(s) - \ddot{\lambda} \lambda \tanh(s)  \right\}  \partial_s -   \dot{\lambda} \cosh^{-2} (s)  \partial_z.
\end{align*}
\end{lemma}
\begin{proof}
This follows directly from the expressions for the coefficients of the dual metric and its determinant in (\ref{determinant}) and (\ref{dual_metric}), and the local coordinate expression for.
\end{proof}

Lemmas \ref{UnitNormal}, \ref{Metric}, \ref{2FF}, \ref{MeanCurvature} and \ref{LaplaceOperator} then collectively prove Proposition \ref{FStarGeometry}.

 \section{Proof of Proposition \ref{StripInvertibilityWithRBC}, Proposition \ref{RBCSolutionsIntegralControl} and Proposition \ref{RBCSolutionsWeightedControl}}
 
 Proposition  \ref{StripInvertibilityWithRBC} follows from the fact that $\tanh(s)$ spans the kernel of $\tilde{\mathcal{L}}$ on $\Lambda(\ell, N)$ with the boundary conditions stated in Proposition  \ref{StripInvertibilityWithRBC} (\ref{SolutionRobinBoundaryCondition}) and (\ref{SolutionNeumanBoundaryCondition}), which we prove below:
 \begin{proposition} \label{WhatsInTheRobinKernel}
 Let $\phi: \Lambda(\ell, N) \rightarrow R$ satisfy:
 \begin{enumerate} 
 \item $\tilde{\mathcal{L}} \phi = 0$. 
\item $\partial_s\phi (\ell, z) \tanh(\ell) - \phi(\ell, z)/\cosh^2(\ell) = 0$.

\item $\partial_z \phi (s, \pm N) = 0 $.
 \end{enumerate}
 Then $\phi$ is a multiple of $\tanh(s)$.
 \end{proposition}
 \begin{proof}
For fixed $s$,  let $\phi_{k} (s)$ denote the Fourier coefficients of $\phi(s, z)$, so that 
\begin{align} \notag
\phi (s, z) =\sum \phi_k(s) e^{ik z}.
\end{align}
Then $\phi_k$ satisfies the equation:
\begin{align} \notag
\tilde{L} \phi_k - k^2 \phi_k = \phi_k'' + (2\cosh^{-2} (s) - k^2) \phi_k = 0.
\end{align}
For each fixed $k$, we have that $\phi_k (0) = 0$ and after possibly re-normalizing, we can assume that $\phi_k(\ell) = \tanh(\ell)$. Assume also that  $\phi_k' (\ell) \leq \tanh'(\ell)$. Thus, for $\ell - s$ sufficiently small and positive, we have that $\phi_k(s) > \tanh(s)$. Let $s_0$ be the largest real number less that $\ell$ so that 
\begin{align} \notag
\phi_k(s_0)  = \tanh(s_0).
\end{align}
Then $w(s) : = \phi_k(s) - \tanh(s)$ is positive on the interior of $[s_0, \ell]$ and vanishes at the endpoints. However, we have
\begin{align}
\tilde{L} w(s)  = \tilde{L} \phi_k = k^2 \phi_k > 0
\end{align}
so that $w(s)$ cannot have an interior maximum, which gives a contradiction.  Thus, we have that 
\begin{align}
\phi_{k}'(\ell) > \tanh' (\ell) = \cosh^{-2} (\ell).
\end{align}
It then immediately follows that if $\phi$ satisfies the Robin boundary condition in the statement, then only the zero mode is present in the Fourier expansion, which gives 
\begin{align} \notag
\phi(s) = c \tanh(s)
\end{align}
and completes the proof.
 \end{proof}
 From this Proposition \ref{StripInvertibilityWithRBC} immediately follows:
 \begin{proof} [proof of Proposition \ref{StripInvertibilityWithRBC}]
The existence of weak solutions and their higher regularity  follows from standard theory. 
 \end{proof}
 
 In order to prove Proposition \ref{RBCSolutionsIntegralControl}, We first observe that the solutions $u_N$ inherit the orthogonality condition from $E$.
 \begin{proposition} \label{SolutionUnweightedOrthogonality}
 It holds that 
 \begin{align} \label{FlatOrthogonalityCondition}
 \int_0^\ell u_N (s, z)  \tanh(s)ds = 0
 \end{align}
 for all $z \in [-N, N]$.
 \end{proposition}
 \begin{proof}
 We have
 \begin{align} \notag
 \partial^2_{z z}u_N + \partial^2_{s s} u_N + 2 \cosh^{-2} (s) u_N =  \partial^2_{z z}u_N + \tilde{L} u_N = E.
 \end{align}
 Multiplying both sides by $\tanh(s)$, integrating by parts and using the Robin boundary condition then gives
 \begin{align} \notag
 \int_0^\ell  \left(\partial^2_{z z}u_N\right)(s, z) \tanh(s) ds =    \partial^2_{z z}\left(\int_0^\ell  u_N(s, z) \tanh(s) ds \right)= 0.
 \end{align}
 The boundary conditions  then immediately imply $\left(\int_0^\ell  u_N(s, z) \tanh(s) ds \right) $ is constant in $z$. The conclusion then follows by adding a multiple of $\tanh(s)$ to $u_N$ if necessary. 
 \end{proof}

\begin{definition} \label{Dim1Energy}
For a function $f$ belonging to the space $W^{1, 2} ([0, \ell])$ we set
 \begin{align} \notag
 e_\ell (f) : = \int_0^\ell f'^2 (s) ds - 2 \int_0^\ell f^2(s) \cosh^{-2} (s) ds - f^2 (\ell) \tanh'(\ell)/\tanh(\ell).
 \end{align}
 \end{definition}

\begin{proposition}[One dimensional weighted Poincare Inequality] \label{PoincareInequality}
There is a universal constant  $\beta > 0$ independent of $\ell$ so that: Let $f \in W^{1, 2}([0, \ell])$ satisfy the orthogonality condition:
\begin{align} \label{WeightedOrthogonalityCondition}
\int_0^\ell f (s)\tanh(s)/\cosh^2(s) ds = 0.
\end{align}
Then it holds that 
\begin{align} \notag
e_\ell(f) \geq \beta \|f : L^2([0, \ell], \cosh^{-2}(s))\|^2.
\end{align}
\end{proposition}

\begin{proof}
Suppose not, and let $(f_k, \ell_k)$ be  a sequence of functions satisfying
\begin{align} \notag
|e_{\ell_k}(f_k)| \leq 1/j \quad \|f : L^2([0, \ell_k], \cosh^{-2}(s))\|^2 = 1.
\end{align}
We then immediately get that $\{ f_k\}$ is a uniformly  bounded sequence in $W^{1, 2}[0, \ell]$ independent of  $\ell$. Assuming $\ell_k \rightarrow \infty$, we then get that $f_k$  strongly sub converges  on compact subsets of $[0, \infty]$ to a limiting function  $f$ in $L^{2}([0, \infty], \cosh^{-2}(s))$. Moroever, the uniform energy bound on the sequence $\{ f_k\}$ gives that 
\begin{align} \notag
|f_k (s)| \leq C s^{1/2}, \quad |f (s)| \leq C s^{1/2}.
\end{align}
The dominated convergence theorem then implies that $f$ satisfying the following conditions:
\begin{align} \notag
e (f) : = e_{\infty} (f) = 0, \quad \|f : L^{2}([0, 1], \cosh^{-2} (s))\| = 1, \quad  \int_0^{\infty} f \tanh(s) / \cosh^2(s) ds = 0.
\end{align}
The Dirichlet condition  $f(0) = 0$ then implies that $f$ is a non-trivial multiple of $\tanh(s)$, which violates that last condition above. This concludes the proof. 
\end{proof}

 \begin{proposition} \label{PoincareInequalityAgain}
 Let $f$ be  a function satisfying the hypothesis of Proposition \ref{PoincareInequality}. Then it holds that
 \begin{align} \notag
 \left( 1 - \frac{2}{2 + \beta} \right)  \int_{0}^\ell f'^2(s)  \leq 2e_\ell (f).
 \end{align}
 \end{proposition}
 \begin{proof}
The inequality in Proposition \ref{PoincareInequality} can be written explicitly as:
\begin{align} \notag
\int_0^\ell f'^2 (s)ds - \int_0^\ell f^2 (2 + \beta) \cosh^{-2} (s) + f^2(\ell)  \tanh'(\ell)/ \tanh(\ell) \geq 0.
\end{align}
Equivalently:
\begin{align} \notag
\left(1 - \frac{2}{2 + \beta} \right)\int_0^\ell f'^2 (s) ds + \left( \frac{2}{2 + \beta} -1 \right)f^2(\ell)  \tanh'(\ell)/ \tanh(\ell) \leq  e(f)
\end{align}
Since $|f^2(\ell)| \leq \ell \int_0^\ell f'^2(s) ds$, the claim follows directly. 
 \end{proof}
There is a discrepancy between the orthogonality condition we would like the solutions $u_N$ to inherit--namely, the weighted orthogonality condition in Proposition \ref{PoincareInequality} (\ref{WeightedOrthogonalityCondition})--and the orthogonality condition we are able to control--that recorded in Proposition \ref{SolutionUnweightedOrthogonality}. Nonetheless, we can show that solutions inherit enough of the weighted orthogonality condition to prove sufficient bounds. 

 \begin{proposition} \label{fOrthogonalPart}
 Let $f(s)$ be a function in  $L^2([0, \ell])$  satisfying the orthogonality condition (\ref{FlatOrthogonalityCondition}) in Proposition \ref{SolutionUnweightedOrthogonality}. Then we can write
 \begin{align} \label{fDecomp}
 f (s) = \alpha \tanh(s)  + g (s)
 \end{align}
 where $g$ satisfies the following conditions
 \begin{align} \notag
\int_0^\ell g(s) \tanh(s) /\cosh^2(s)  = 0,  \quad   \int_0^\ell f^2 (s) ds \leq \int_0^\ell g^2 (s) ds.
 \end{align}
 \end{proposition}
 \begin{proof}
Write
\begin{align} \notag
f = \alpha \tanh(s) +  (f - \alpha \tanh(s))  := \tanh(s) + g
\end{align}
Then choosing $\alpha$ appropriately, it is clear that we can arrange for $g$ to be orthogonal to $\tanh(s)/\cosh^2(s)$. In particular, we choose $\alpha$ so that 
\begin{align}\label{AlphaExplicit}
\alpha = \frac{\int_0^\ell f(s) \tanh(s)/\cosh^2(s) ds}{\int_0^\ell \tanh^2(s)/ \cosh^2(s) ds}.
\end{align}
\end{proof}

\begin{proposition} \label{UnweightedPoincare}
 Let $f$ and $g$ be as in the statement of Proposition \ref{fOrthogonalPart}. Then it holds that 
 \begin{align} \notag
 e_\ell(f)  = e_\ell (g) \geq \left( 1 - \frac{2}{2 + \beta}\right)\left( \int_0^{\ell} g'^2(s) ds  \right)
\end{align}
\end{proposition}

\begin{proof}
We have that
\begin{align} \notag
e_\ell(f) = B [f, f].
\end{align}
It is then clear that the right hand side above is a bilinear form and we have
\begin{align} \notag
e_\ell(f)  &  = B[ \alpha \tanh(s) + g,  \alpha \tanh(s) + g ]\\ \notag
& = \alpha^2 B[ \tanh(s),    \tanh(s) ]  + B[ g,   g ]  + 2 \alpha B[ g,   \tanh(s) ]  \\ \notag
& =  B[g,  g  ] \\ \notag
& = e_\ell (g).
\end{align}
where the last equality above follows from the definition of $e(-)$ and $B[ -, - ]$ and where the third line follows from the second since $\tanh(s)$ is in the kernel of $\tilde{L}$ with the boundary conditions  (\ref{SolutionRobinBoundaryCondition}) and (\ref{SolutionNeumanBoundaryCondition}).
\end{proof}

We are now ready to prove Proposition \ref{RBCSolutionsIntegralControl}:

\begin{proof}[Proof of Proposition \ref{RBCSolutionsIntegralControl}]
In the following, in order to conveniently reference  Propositions \ref{PoincareInequality}, \ref{PoincareInequalityAgain} and \ref{UnweightedPoincare} while using their notational conventions, we set
\begin{align} \notag
f (s,z) = u_N(s, z) = \alpha(z) \tanh(s) + g(s, z).
\end{align}
where the multiple $\alpha$ is chosen so that 
\begin{align} \notag
\int_0^{2 \ell} g(s,z) \tanh(s) ds = 0.
\end{align}
Additionally, we will throughout the proof abbreviate $\Lambda : = \Lambda ( \ell, N)$, $\Lambda_0 : = \Lambda( \ell, 2 \pi)$ and we will use `` $'$  '' to denote derivation with respect to $s$. Multiplying the left hand side  of  (\ref{StripSolEq}) by $f$, integrating in $s$ from $0$ to $\ell$ and using the boundary conditions gives:
\begin{align} \notag
I(z)  & :  = \int_0^{\ell} f(s, z) \partial^2_{z z} f (s, z) ds - \int_0^\ell \left(\partial_s f \right)^2 ds  + 2 \int_0^{\ell} f^2 (s, z) \cosh^{-2} (s)ds \\ \notag
&  \quad  + f^2(\ell, z) \tanh'(\ell)/ \tanh( \ell) \\ \notag
& =  \int_0^{\ell} f(s, z) \partial^2_{z z} f (s, z) ds  - e_{\ell} (f(z,- )).
\end{align}
 where the energy $e_{\ell}(-)$ is defined in Definition \ref{Dim1Energy}. Recall that from Proposition  \ref{UnweightedPoincare} we have $ e_{ \ell}(g(z, -))  =  e_{\ell} (f(z, -))$. Proposition \ref{PoincareInequalityAgain} then gives:
\begin{align}  \notag
\left(1 - \frac{2}{2 +\beta} \right) \int_{\Lambda}  g'^2 &  \leq \left| \int_{\Lambda_0} f E  \right|  \\ \notag
& \leq \left( \int_{\Lambda_0}f^2  \right)^{1/2} \left(  \int_{\Lambda_0} E^2   \right)^{1/2} \\ \notag
& \leq \left( \int_{\Lambda_0}g^2  \right)^{1/2} \left(  \int_{\Lambda_0} E^2   \right)^{1/2} \\ \notag
& \leq C \ell \left( \int_{\Lambda_0} g'^2  \right)^{1/2} \left(  \int_{\Lambda_0} E^2   \right)^{1/2}. 
\end{align}
Summarizing the above estimates:
\begin{align} \notag
\int_{\Lambda}  g'^2 \leq C \ell^2 \int_{\Lambda_0} E^2.
\end{align}
We then have
\begin{align} \notag
\int_\Lambda f^2 ds \leq \int_\Lambda g^2 ds \leq C \ell^2 \int_\Lambda g'^2  \leq C \ell^4 \int_{\Lambda} E^2.
\end{align}
Applying H{\" o}lder's inequality to the above inequality then gives:
\begin{align} \notag
\left(\int_\Lambda |f| \right)^2 \leq \ell \int_\Lambda f^2 ds  \leq C \ell^{5}  \int_{\Lambda} E^2.
\end{align}
Recalling the expression for $\alpha$ in (\ref{AlphaExplicit}), we the conclude that 
\begin{align} \notag
|\alpha| \leq C  \int_0^{ \ell} |f|  \leq C \ell^{3/2} \left(\int_{\Lambda} E^2 \right)^{1/2}.
\end{align}
It then follows that 
\begin{align} \notag
\|f': L^1\| \leq \alpha \| \cosh^{-2} (s) : L^1 \| + \| g' : L^1\| \leq C \ell^3 \| E : C^{0, \alpha} (\Lambda_0, 1) \|.
\end{align}
This concludes the proof.

\end{proof}
 
The $L^2$ estimate for the solutions $u_N$ in Proposition \ref{RBCSolutionsIntegralControl} and a maximum principle then immediately translate into existence function $u_\infty$ satisfying the weighted estimate in Proposition \ref{RBCSolutionsWeightedControl}. 

\begin{proposition} \label{LimitSolutions}
There is a function $u_\infty: \Lambda(\ell, \infty) : \rightarrow  R$ such that
\begin{enumerate}
\item \label{LimitSolutionsConvergence} For any compact subset $K \subset \Lambda(\ell, \infty)$, the functions $u_{N}$ converge to $u_{\infty}$ smoothly in $C^{2, \alpha} (K)$. 

\item \label{LimitSolutionsBasic} The function $u_\infty$ satisfies the boundary value problem:
\begin{align} \notag
\tilde{\mathcal{L}} u_\infty = E, \quad u_\infty(0, z) = 0, \quad u_{\infty}' (\ell, z)\tanh(\ell) - u(\ell, z)/\cosh^{-2} (\ell) = 0.
\end{align}

\item \label{LimitSolutionsOrthogonality} $u_\infty$ also satisfies the orthogonality condition
\begin{align} \notag
\int_0^\ell u_\infty(s, z) \tanh(s) ds = 0, \quad \forall z \in R.
\end{align}

\item \label{LimitSolutionsL2} $u_\infty$ satisfies the  estimate
\begin{align} \notag
\| u_\infty : L^1(\Lambda(\ell, \infty))\|, \quad \| \partial_su_\infty : L^1(\Lambda(\ell, \infty)\| \leq C \ell^3 \| E : C^{0, \alpha} (\Lambda(\ell, 2\pi), 1) \|.
\end{align}
\end{enumerate}
\end{proposition}
\begin{proof}
This is a direct consequence of the uniform $L^2$ estimates recorded for the functions $u_N$ in Proposition \ref{RBCSolutionsIntegralControl} and standard techniques.
\end{proof}

We record the maximum principle below:

\begin{proposition} \label{LimitSolutionsMaximumPrinciple}
The function  
\begin{align} \notag
\bar{S}(w) : = \sup_{(s, z) \in\Lambda(\ell, \infty),  z \geq w}  |u_{\infty} (s, z)|
\end{align}
is monotonically decreasing for $w > 2 \pi$. 
\end{proposition}

\begin{proof}
Suppose  there is $(s_0, z_0) \in \Lambda(\ell, \infty)$ so that $u_\infty(s, z)$ obtains a positive maximum  on $\Lambda^+ : = \Lambda(\ell, \infty) \cap\{ z \geq w  \}$ and assume that  $(s_0, z_0)$ is in the interior of $\Lambda^+$.  Then there is a multiple $m$ of $\tanh(s)$ and $p$ in the interior of  $\Lambda^+$ so that  $g (s, z) : = m \tanh(s) - u_{\infty} (s, z)$ is positive on  $\Lambda^+ \setminus \{ p\} $ and so that $g (p) = 0$, which gives a contradiction. Thus, $(s_0, z_0)$ must be on the boundary of $\Lambda^+$. Suppose now that $(s_0, z_0) = (\ell, z_0)$ for $z_0 > w$. We again choose $m$ so that 
\begin{align} \notag
 g (s, z) : = m \tanh(s) - u_{\infty}(s, z)
 \end{align}
 has the property that 
\begin{align} \notag
g(\ell, z_0) = 0, \quad g(\ell, z) \geq 0.
\end{align}
Observe that $g$ satisfies the Robin boundary condition (\ref{SolutionRobinBoundaryCondition}) at $s = \ell$ and we have
\begin{align} \notag
\Delta g + 2 \cosh^{-2}(s) g= g_{z z} + g_{ss} +   2 \cosh^{-2}(s) g =  0.
\end{align}

At $(\ell, z_0)$ we then get
\begin{align} \notag
g_s \geq 0, \quad g_{s s} < 0.
\end{align}
Thus, $g$ has an interior minimum at some point $p \in \Lambda^+$ and is non-negative  on the boundary of $\Lambda^+$, which gives a contradiction. We then conclude that if $\bar{S}(w)$ is finite, then it holds that 
\begin{align} \notag
\bar{S} (w)  = |u_\infty (s_w, w)|
\end{align}
for some $s_0 \in [0, \ell]$, from which it follows that $\bar{S} (w)$ is either monotonically decreasing or increasing in $|w|$. To see that  $\bar{S}(w)$ is a monotonically decreasing in $|w|$, observe that 
\begin{align} \notag
\bar{S} (w) = |u_\infty(s_w, w)| \leq \int_0^\ell  \left|u'_\infty(s, w)  \right| ds
\end{align}
Integrating in $w$ then gives that 
\begin{align}
\int_{-\infty}^\infty \bar{S} (w) \leq \int_{\Lambda(\ell)}  \left|u'_\infty  \right| \leq C \ell^3 \| E : C^{0, \alpha} (\Lambda(\ell, 2\pi), 1) \|.
\end{align}
 This completes the proof of Proposition \ref{LimitSolutionsMaximumPrinciple}.
\end{proof}

We can now prove Proposition \ref{RBCSolutionsWeightedControl}:

\begin{proof}[Proof of Proposition \ref{RBCSolutionsWeightedControl}]
In the proof of Proposition \ref{LimitSolutionsMaximumPrinciple} above, we saw that the function $\bar{S} (w)$ is an integrable function of $w$ on $(-\infty, \infty)$, which is monotonic on the the sets $(-\infty, 2 \pi)$ and $(2\pi, \infty)$. It then immediately follows that $\bar{S}(w)$ satisfies the bound
\begin{align} \notag
|\bar{S}(w)| \leq C/(1 + |w|).
\end{align}
In fact,  a stronger estimates hold, but this will suffice for our purposes. 
\end{proof}

 \section{Proof of Proposition \ref{MainInvertibilityStatement}}
 
 Proposition \ref{MainInvertibilityStatement} is a  direct consequence of Proposition \ref{RBCSolutionsWeightedControl} and the convex pinching property of the scale functions. 

\begin{proposition}\label{ZGoodInterval}
There is a constant $A$ so that, given $z_0 \in [0, 1]$, then for all $z \in [0, 1]$ with 
\begin{align} \notag
|z - z_0| \leq A \lambda^{\epsilon} (z_0)
\end{align}
it holds that 
\begin{align}
\left|\frac{\lambda(z)}{\lambda(z_0)} - 1 \right| \leq 1/2.
\end{align}
\end{proposition} 
 
 \begin{proposition} \label{SigmaGoodInterval}
 For $\sigma_0 \in (0, 1)$, let $\delta = \delta (\sigma_0)$ be maximal so that: For $\sigma \in (0, 1)$ satisfying $|\sigma- \sigma_0| \leq \delta$ it holds that 
 \begin{align} \notag
 \left| \frac{\lambda (\sigma)}{\lambda(\sigma_0)}  -1 \right| \leq 1/2.
 \end{align}
 Then we have the following estimate:
 \begin{align} \notag
 \delta(\sigma_0)  \geq 3 \lambda^{1 - \epsilon} (\sigma_0). 
 \end{align}
 \end{proposition}
 \begin{proof}
 For $\sigma$ as in the statement of the proposition, we have
 \begin{align} \notag
 |\lambda (\sigma) - \lambda(\sigma_0)| & = \left| \int_{\sigma_0}^\sigma \dot{\lambda} (\sigma') d \sigma' \right|  \\ \notag
 & \leq \delta \sup_{\sigma} |\dot{\lambda} (\sigma)| \\ \notag
 & \leq 3/2\delta (\lambda (\sigma_0))^\epsilon,
 \end{align}
 where in the last line above we have used the convex pinching assumptions in (\ref{ScaleFunctionAssumptions}). With $\sigma = \sigma_0 + \delta$ and using the maximality of $\delta$ we get that 
 \begin{align} \notag
  |\lambda (\sigma) - \lambda(\sigma_0) | = \lambda(\sigma_0)/2 \leq 3/2 \delta \lambda^\epsilon(\sigma_0).
 \end{align}
 It then follows directly that 
 \begin{align} \notag
 \delta \geq 3 \lambda^{1 - \epsilon} (\sigma_0).
 \end{align}
This completes the proof.
 \end{proof}

\begin{proof}[Proof of Proposition \ref{ZGoodInterval}]
Let $\sigma_0$ and $\delta = \delta(\sigma_0)$ be as in Proposition \ref{SigmaGoodInterval} and set $z_0 = z(\sigma_0)$, $z = z (\sigma)$. Then for $|\sigma - \sigma_0| \leq \delta$ we have from Proposition \ref{SigmaGoodInterval} that
\begin{align} \notag
\left|\frac{\lambda (z)}{\lambda(z_0)} - 1 \right| \leq 1/2.
\end{align} 
\begin{align} \notag
z - z_0  = \int_{\sigma_0}^\sigma \frac{d \sigma'}{\lambda(\sigma')} =  \frac{1}{\lambda(\sigma_0)} \int_{\sigma_0}^\sigma \frac{\lambda (\sigma_0)}{\lambda(\sigma')} d \sigma' \geq 2/3 \frac{|\sigma - \sigma_0|}{\lambda(\sigma_0)}
\end{align}
with $\sigma = \sigma_0 + \delta$, we then get
\begin{align} \notag
|z - z_0|& \geq 2/3 \frac{\delta}{\lambda(\sigma_0)}  \\ \notag
& \geq 2/3 \frac{\lambda^{1 - \epsilon} (\sigma_0)}{\lambda(\sigma_0)} \\ \notag
& \geq 2/3 \lambda^{\epsilon} (\sigma_0).
\end{align}
This completes the proof.
\end{proof}

\begin{definition}
We set
\begin{align} \notag
z_j : = j \pi, \quad \lambda_j : = \lambda(z_j), \quad \ell_j : = \underline{\ell} (z_j) \quad j \in \mathbb{N}.
\end{align}

Additionally, we  let $\Lambda_j$ be the domain given as follows:
\begin{align} \notag
\Lambda_j : = \{ (s, z): |z| \leq 2 \pi, |s| \leq 2 \ell_j\}.
\end{align}
We also let $\{\psi_j(z)\}$ be a smooth partition of unity on $R$ such that 
\begin{enumerate}
\item It holds that 
\begin{align} \notag
\psi_j(z + \pi) = \psi_{j + 1} (z)
\end{align}
\item $\psi_0(z)$ is an even function of $z$.
\item The support of $\psi_0$ is contained in the interval $[-3/2\pi, 3/2 \pi]$.
\end{enumerate}
\end{definition}

\begin{proposition} \label{PartitionedSolutionPropogated}
Set $E_j(s, z): = \psi_j (z) E(s, z)$.  Then there is a function $v_j: \Lambda_j \rightarrow R$ so that 
\begin{enumerate}
\item It holds that 
\begin{align} \notag
\tilde{\mathcal{L}} v_j = E_j.
\end{align}
\item  \label{PartitionedSolutionEstimate} There is universal constant $K$ (depending only of the $C^{2, \alpha}$ norm of $\psi_0$ and the invertbility constant in Proposition \ref{RBCSolutionsWeightedControl}) so that: The function $v_j$ satisfies the weighted Holder estimate:
\begin{align} \notag
\left\| v_j: C^{2, \alpha} \left(\Lambda(2 \ell_j), \frac{1}{1 + |z -z_j|}\right) \right\| \leq K \lambda^{\epsilon_1}_j,
\end{align}
where $\epsilon_1 : = (1 - 10 \tau) \epsilon_0$.
\end{enumerate}
\end{proposition} 
Before proving Proposition \ref{PartitionedSolutionPropogated}, we observe a few facts. 

\begin{proposition} \label{DualKernelFunction}
It holds that 
\begin{align} \notag
\tilde{L} (s \tanh(s) - 1) = 0.
\end{align}
\end{proposition}

\begin{definition} \label{CoFunction}
Let $\varphi (s): R \rightarrow R$ denote a fixed smooth function such that 
\begin{enumerate}
\item For $s \geq 2$ it holds that $\varphi (s) = s \tanh(s) - 1$.
\item $\varphi(s)$ is an odd function of $s$: $\varphi (s) = - \varphi (-s)$.
\end{enumerate}
\end{definition}

\begin{proposition} \label{ProjectingFunctions}
Let $\psi(z)$ be fixed smooth cutoff function with support on the interval $[-2 \pi, 2\pi]$ and set
\begin{align} \notag
\bar{f}(s, z) = \psi (z) \varphi (s), \quad \bar{g}(s,z) = z\psi (z) \varphi(s). 
\end{align}
Then the following statements hold:
\begin{enumerate}
\item For $\ell$ sufficiently large it holds that 
\begin{align} \notag
\left| \int_{R^2} \left( \tilde{\mathcal{L}} \bar{f}\right) \tanh(s) \right|  > 1/2,  \quad \left| \int_{R^2} \left( \tilde{\mathcal{L}} \bar{g}\right)  z \tanh(s) \right| >  1/2 .
\end{align}
\item It holds that 
\begin{align} \notag
\int_{R^2} \left( \tilde{\mathcal{L}} \bar{f}\right)  z \tanh(s)  = \int_{R^2} \left( \tilde{\mathcal{L}} \bar{g}\right)  \tanh(s) = 0.
\end{align}

\end{enumerate}
\end{proposition}

\begin{proof}
Computing directly gives
\begin{align} \notag
\tilde{\mathcal{L}}\bar{ f} = \psi'' \varphi + \psi \tilde{L} \varphi = I + II.
\end{align}
We then have that 
\begin{align} \notag
\int_{R^2} I \tanh(s) & = \int_{0}^{2\ell} \varphi(s) \tanh(s) \left( \int_{-2 \pi}^{2 \pi}  \psi'' (z) dz \right) ds = 0.
\end{align}
Additionally, we have
\begin{align} \notag
\int_{R^2} II \tanh(s) & = \int_{- 2\pi}^{2\pi} \psi(z) \int_{0}^{2 \ell} (\tilde{L} \varphi) \tanh(s) ds \\ \notag \\ \notag
&=  \int_{- 2\pi}^{2\pi} \psi(z) \int_{0}^{2 \ell} \left( \varphi' (2 \ell) \tanh(2 \ell) - \varphi(2 \ell) \cosh^{-2} (2 \ell) \right) ds.
\end{align}
Similarly, we have 
\begin{align} \notag
\int_{R^2}  \left(\tilde{\mathcal{L}} \bar{g} \right) z \tanh(s) & = \int_0^{2 \ell} \partial_z \bar{g}(s,  2 \pi) 2 \pi \tanh(s) - \bar{g}(s, 2 \pi) \tanh(s) ds  - \int_0^{2 \ell}  \partial_z \bar{g}(s,  -  2 \pi)(- 2 \pi) \tanh(s) - \bar{g}(s,  - 2 \pi) \tanh(s) ds \\ \notag
& \quad + \int_{- 2\pi}^{2 \pi} \psi(z)z^2 \left( \varphi' (2 \ell) \tanh(2 \ell) - \varphi(2 \ell) \cosh^{-2} (2 \ell) \right)  dz \\ \notag
& = \int_{- 2\pi}^{2 \pi} \psi(z)z^2 \left( \varphi' (2 \ell) \tanh(2 \ell) - \varphi(2 \ell) \cosh^{-2} (2 \ell) \right)  dz  \\ \notag
& > 0 .
\end{align}
The remaining claims are then  directs consequences of the symmetries that the functions involved. 
\end{proof}

\begin{proposition} \label{PreppedErrorTerm}
The following statements hold:

\begin{enumerate}
\item \label{PET1} The function $E(s, z_j + z)$ is supported on  $\Lambda(2 \ell_j,  2 \pi)$.
\item\label{PET2} There are unique multiples $\bar{a}_j$ and $\bar{b}_j$ so that the function 
\begin{align} \notag
\hat{E}_j (s, z): =  E_j (s, z_j + z)- \bar{a}_j \left( \tilde{\mathcal{L}}\bar{f} \right) (s, z)- \bar{b}_j  \left( \tilde{\mathcal{L}}\bar{g} \right)(s, z)
\end{align}
 satisfies the following orthgonality conditions:
\begin{align} \notag
\int_{\Lambda(2 \pi, 2 \ell)} \hat{E}_j (s, z) \tanh(s) ds =  \int_{\Lambda(2 \ell_j 2 \pi)} \hat{E} (s ,z ) z \tanh(s) ds = 0.
\end{align}
\item \label{PET3} The multiples $\bar{a}_j$ and $\bar{b}_j$ satisfy the estimates
\begin{align} \notag
|\bar{a}_j|, |\bar{b}_j| \leq C \ell_j \|E_j : C^{0, \alpha} (\Lambda(2 \pi, 2 \ell_j), 1)\| \leq C \ell_j \lambda_j^{\epsilon_0}.
\end{align}

\item \label{PET4} It holds that 
\begin{align} \notag
\| \hat{E}_j : C^{0, \alpha} (\Lambda(2 \pi, 2\ell), 1)\| & \leq \|E_j : C^{0, \alpha} (\Lambda(2 \pi, 2 \ell), 1)\| + \bar{a}_j \|\bar{f}: C^{0, \alpha} (\Lambda(2 \pi, 2\ell), 1) \|  \\ \notag
& \quad + \bar{b}_j \|g : C^{0, \alpha} (\Lambda(2 \pi, 2\ell), 1)\| \\ \notag
& \leq C\lambda_j^{\epsilon_0} \left( 1 + \ell_j^2 \right)
\end{align}
\end{enumerate}
\end{proposition}

\begin{proof}
(\ref{PET1}) is clear. (\ref{PET2})  and (\ref{PET3}) and (\ref{PET4}) are  direct consequences of Proposition \ref{ProjectingFunctions}.
\end{proof}

\begin{proof}[ Proof of Proposition \ref{PartitionedSolutionPropogated}]
We wish to apply  Proposition \ref{RBCSolutionsWeightedControl}.  In order to do this directly, our error term must satisfy the strong orthogonality condition along lines stated in Proposition \ref{StripInvertibilityWithRBC}. Set 
\begin{align} \notag
a(z) : = \left( \int_{0}^{2 \ell_j} \hat{E}(s, z) \tanh(s)ds \right)/ \left( \int_0^{2 \ell_j} \tanh^2(s) ds\right), \quad b(z) : =  z a (z)
\end{align}
Observe that from Proposition \ref{PreppedErrorTerm} we have:
\begin{align} \notag
\int_{- 2 \pi}^{2\pi} a(z)dz  = \int_{\Lambda(2\pi, 2 \ell_j)}\hat{E}(s, z) \tanh(s) = 0, \quad  \int_{- 2 \pi}^{2\pi} b(z) dz = \int_{\Lambda(2\pi, 2 \ell_j)}\hat{E}(s, z)  z\tanh(s) = 0.
\end{align}
Set
\begin{align} \notag
A(z)  : = \int_{-\infty}^z \int_{-\infty}^{z'} a(z'') dz''  = z\int_{-\infty}^z a(z')dz' - \int_{-\infty}^z b(z') dz'
\end{align}
where the second equality follows from integration by parts. We then put
\begin{align} \notag
h (s, z) : =  A(z) \tanh(s)
\end{align}
It then follows directly that $h(s, z)$ is supported on the set $\{|z| \leq 2 \pi \}$ and that 
\begin{align} \notag
\tilde{\mathcal{L}} h (s, z) = a(z) \tanh(s).
\end{align}
Moreover, by construction, we have that 
\begin{align} \notag
\int_0^{2 \ell_j}\left(\hat{E}(s, z) - a (z) \tanh(s) \right) \tanh(s) ds = \int_{0}^{2 \ell_j} \hat{E} (s, z)\tanh(s) ds - a(z)\int_0^{2 \ell} \tanh^2(s) ds = 0.
\end{align}
The function
\begin{align} \notag
 F(s ,z) : = \hat{E} (s, z) - a (z) \tanh(s)
\end{align}
 then satisfies the bound
\begin{align} \notag
\| F : C^{0, \alpha} (\Lambda(2 \pi, 2 \ell_j)) \|  \leq  C\lambda_j^{\epsilon_0} (1 + \ell_j^2) \leq \lambda_j^{(1 - 4 \tau) \epsilon}
\end{align}
We then apply Proposition \ref{RBCSolutionsWeightedControl} to obtain a function $v$ satisfying:
\begin{align} \notag
\tilde{\mathcal{L}} v (s, z) = F  (s, z) =   E(s, z + z_j ) - \bar{a} \left( \tilde{\mathcal{L}}\bar{f} \right)(s,z) - \bar{b} \left( \tilde{\mathcal{L}}\bar{g} \right) (s, z) - a(z) \tanh(s),
\end{align}
and the claimed bounds. We conclude by setting
\begin{align} \notag
v_j (s, z) :   = v(s, z - z_j) +  \bar{a}\bar{f} (s, z - z_j) + \bar{b} \bar{g}(s, z - z_j) + A (z - z_j) \tanh(s).
\end{align}
\end{proof}

\begin{definition} \label{DualDomains}
We let $\Lambda^*_j$ be the domain given by:
\begin{align}  \notag
\Lambda_j^* : = \{(s, z): |s| \leq 2 \ell, |z - z_j| \leq \rho_j\}
\end{align}
where  where  $\rho_j : = A \lambda^{-\epsilon}_j/6$.  Additionally, we let $\psi^*_j$(z) denote the cutoff function given by:
\begin{align}
\psi^*_j(z) : = \psi_0[\rho_j, \rho_j/2] (|z|).
\end{align}
\end{definition}

\begin{proposition} \label{DualDomainProps}
The following statements hold:
\begin{enumerate}
\item The domain $\Lambda \cap \{ |z - z_j| \leq \rho_j\}$ is contained in $\Lambda^*_j$ and the boundary curve $\partial ( \Lambda \cap \{ |z - z_j| \leq \rho_j\} )$ is strictly separated from the boundary curve $\{s = 2 \ell \}$ on $\Lambda^*_j$.

\item The functions $\psi^*_j$ are smooth non-negative functions with support contained in $\Lambda^*_j$  and it holds that 
\begin{align} \notag
|\nabla^{(k)} \psi_j^* | \leq C \rho_j^k.
\end{align}

\item The support of the  gradient of $\psi_j^*$ is contained in $\Lambda^*_j \setminus \Lambda_j$.
\end{enumerate}
\end{proposition}
\begin{proof}
All claims in the proposition are simple consequences of Definition \ref{DualDomains},  the definition of $\lambda$ in (\ref{LambdaDef}) and the properties of the scale function $\lambda$ in (\ref{ScaleFunctionAssumptions}).
\end{proof}

\begin{definition}
The functions $v_j^*$ and $E_j^*$ are given as follows:
\begin{align}
v^*_j (s, z): = \psi_j^*(z)v_j (s, z), \quad E_j^* : = E_j - \tilde{\mathcal{L}} (\psi_j^*v_j)
\end{align}
\end{definition}
\begin{proposition} \label{DualObjectProperties}

The following statements hold:
\begin{enumerate} \notag
\item \label{LocalizedSupports}The supports of $v^*_j$ and $E^*_j$  are contained in $\Lambda_j^*$,
\item \label{LocalizedSupportIntersect} Assume that $z_k$ and $z_j$ are such that  $|z_k - z_j| \geq  A \lambda^{-\epsilon}_j$, where $A$ is as in Proposition \ref{ZGoodInterval} . Then the supports of $v^*_j$ and $v_k^*$  do not intersect. Likewise, the supports of  $E^*_j$ and $E^*_k$ do not intersect under the same conditions. 

\item \label{DualSolutionEstimate} The functions $v^*_j$ satisfy the estimates
\begin{align}
\left\| v^*_j : C^{2, \alpha} \left(\Lambda_j^*, \frac{1}{1 + |z - z_j|} \right) \right\| \leq   C\gamma \lambda_j^{\epsilon_1}.
\end{align}

\item \label{DualErrorEstimate} The functions $E^*_j$ satisfies the estimate
\begin{align}
\| E^*_j : C^{0, \alpha} \left(\Lambda_j^*, \frac{\rho_j }{1 + |z - z_j|} \right) \| \leq   C \gamma  \lambda_j^{\epsilon_1}
\end{align}

\end{enumerate}
\end{proposition}

\begin{proof}
Statement (\ref{LocalizedSupports}) is a direct consequence of the definition of $\psi^*_j$ in Definition \ref{DualDomains}. We prove statement (\ref{LocalizedSupportIntersect}) in two cases. Let $k \in \mathbb{N}$ be such that 
\begin{align} \notag
\lambda_k \leq \frac{1}{6} \lambda_j.
\end{align} 
By construction, it holds that for all $z $ belonging to the interval $(z_k - \rho_k, z_k + \rho_k)$ we have that 
\begin{align} \notag
\lambda(z) \leq 3/2 \lambda_k \leq 1/4 \lambda_j.
\end{align}
Again, by construction we have that for $z $ belonging to  $(z_j - \rho_j, z_j + \rho_j)$, it holds that 
\begin{align} \notag
\lambda(z) \geq 1/2 \lambda_j.
\end{align}
Thus, if the two intervals intersect, we obtain
\begin{align} \notag
\lambda (z) \geq  2 \lambda(z),
\end{align}
which is a contradiction.  Now, assume that $k \in \mathbb{N}$ is such that 
\begin{align} \notag
\lambda_k \geq \frac{1}{6} \lambda_j.
\end{align}
and so that $z_k$ does not belong to the interval $(z_j - \rho_j, z_j + \rho_j)$.  We then get that 
\begin{align} \notag
\rho_k \leq 6^\epsilon \rho_j \leq 2 \rho_j.
\end{align}
for $\epsilon > 0$ sufficiently small.  Suppose now that $z \in R$ satisfies
\begin{align} \notag
|z_j - z| \leq \rho_j, \quad |z_k  - z| \leq \rho_k.
\end{align} 
Then from the triangle inequality we get that 
\begin{align} \notag
|z_k - z_j| \leq |z_k - z| + |z_j - z| \leq \rho_k + \rho_j \leq  3 \rho_j \leq (1/2) A \lambda_j^{- \epsilon}.
\end{align}
However,  by assumption we have that 
\begin{align} \notag
|z_k - z_j| \geq A \lambda_j^{- \epsilon}.
\end{align} 
This yields a contradiction and gives the proof of  Statement (\ref{LocalizedSupportIntersect}). Statements  (\ref{DualSolutionEstimate}) and  (\ref{DualErrorEstimate}) follow directly for the weighted estimate for $v_j$ in Proposition \ref{PartitionedSolutionPropogated} (\ref{PartitionedSolutionEstimate}) and the definition of $\psi^*_j$.
\end{proof}
\begin{proposition} \label{FirstIterationFunctions}
Set 
\begin{align} \notag
v^* : = \sum_j v^*_j, \quad E^* : = \sum_j E^*_j.
\end{align}
Then the following statements hold:
\begin{enumerate}
\item \label{LocallyFiniteSums} The infinite sums defining $v^*$ and $E^*$ are locally finite and thus converge on compact subsets of $\Lambda$. 
\item \label{VStarEstimate} The function $v^*$ satisfies the estimate
\begin{align} \notag
\| v^*: C^{2, \alpha} (\Lambda, \lambda^{\epsilon_2}) \| \leq C \gamma
\end{align}
where $\epsilon_2 = (1 -2 0 \tau) \epsilon$
\item \label{EStarEstimate} Given any constant $\bar{\delta} > 0$, there is $C_0$ in (\ref{ScaleFunctionAssumptions}) so that:
\begin{align} \notag
\| E^*: C^{0, \alpha} (\Lambda, \lambda^{\epsilon_0})\| \leq   \bar{\delta} \gamma.
\end{align}
\end{enumerate}
\end{proposition}
\begin{proof}
Fix a $j \in \mathbb{N}$. From Proposition \ref{DualObjectProperties} (\ref{LocalizedSupportIntersect}), we have on  $\Lambda_j$ that:
\begin{align} \notag
\left\|\sum_{k} v^*_k \right\|_{2,\alpha}& = \sum_{ |z_k - z_j| \leq  6 \rho_j}\left\| v^*_k \right\|_{2, \alpha} \\ \notag
 & \leq  C \gamma \sum_{ |z_k - z_j| \leq  6\rho_j} \frac{\lambda_k^{\epsilon_1}}{1 + |z_k - z_j|} \\ \notag
 & \leq C \gamma \lambda^{\epsilon_1}_j  \int_{- 6 \rho_j}^{6 \rho_j} \frac{1}{1 + |w| }dw \\ \notag
 & \leq  C  \gamma\lambda_j^{\epsilon_1} \log (\rho_j)\\ \notag
\end{align} 
Setting $\epsilon_2 =  ( 1 -  20 \tau) \epsilon_0$   gives claim (\ref{VStarEstimate}). Claim (\ref{EStarEstimate}) follows similarly:  On $\Lambda_j^*$ we  again have
\begin{align} \notag
\sum_{k} E^*_k & = \sum_{ |z_k - z_j| \leq \rho_j} E^*_k \\ \notag
&  \leq C\gamma   \sum_{ |z_k - z_j| \leq \rho_j} \frac{\rho_k\lambda_k^{\epsilon_1}}{1 + |z_k - z_j|} \\ \notag
& \leq C\gamma \rho_j \lambda_j^{\epsilon_1} \int_{-6 \rho_j}^{6 \rho_j}  \frac{dw}{1 + w } \\ \notag 
& \leq C \gamma \rho_j \lambda_j^{ \epsilon_2} \\ \notag
\end{align}
Assuming that $1 - 20 \tau > 0$, we can take $\lambda_j$ sufficiently small so that 
\begin{align} \notag
C \lambda_j^{(1 -20 \tau) \epsilon}\leq \bar{\delta},
\end{align}
which gives the claim.
\end{proof}
Proposition \ref{MainInvertibilityStatement} now follows immediately.
\begin{proof}[Proof of Proposition \ref{MainInvertibilityStatement}]
Let $E$ be a function as in the statement of the proposition and set
\begin{align} \notag
\gamma : = \| E : C^{0, \alpha} (\Lambda, \lambda^{\epsilon_0}) \|.
\end{align}
We then apply Proposition \ref{FirstIterationFunctions} to obtain functions $v^*$, $E^*$ satisfying the following estimates:
\begin{align} \notag
\| v^*: C^{2, \alpha} (\Lambda, \lambda^{\epsilon_2}) \| \leq C \gamma, \quad  \| E^*: C^{0, \alpha} (\Lambda, \lambda^{\epsilon_0})\| \leq   \bar{\delta} \gamma.
\end{align}
We then repeat the process to obtain a sequence of functions $(v_i, E_i)$ satisfying
\begin{align} \notag
\tilde{\mathcal{L}} v_i = E_i
\end{align}
and the  following estimates:
\begin{align} \notag
 \| E_i : C^{2, \alpha} (\Lambda, \lambda^{\epsilon_0}) \| \leq \bar{\delta}^i \gamma, \quad  \| v_i: C^{2, \alpha} (\Lambda, \lambda^{\epsilon_2}) \|  \leq (C\bar{\delta})^i \gamma. 
\end{align}
It is then straightforward that for $\bar{\delta}$ sufficiently small, the partial sums $\sum_{i = 0}^k v_i$  goes to infinity to a  function $v$ satisfying all the claims in the Proposition.  This completes the proof
\end{proof}

\section{The non-linear problem}

\begin{proof}[Proof of Proposition \ref{FStarGraphMC}]
From Proposition (\ref{FStarGraphMC}) it follows that  for $\lambda$ sufficiently small we have:
\begin{align} \notag
\mathfrak{a}[F^*(s, z)] \geq 1/2,  \quad |\nabla F(s, z)| > C \lambda(z) \cosh (s), \quad \ell_{j, \alpha} [F(s, z)] \leq C_0
\end{align}
and that the vector field $\mathcal{E} (s, z): = \lambda(z) u(s, z)\nu^*(s, z)$ satisfies
\begin{align} \notag
\|\mathcal{E} \|_{j, \alpha} (s, z) \leq C \lambda (z) \| u\|_{j, \alpha} (s, z) \leq C \epsilon |\nabla F(s, z)|.
\end{align}
The claim then follows directly from Proposition \ref{HQEstimates} and taking $\epsilon$ sufficiently small. 
\end{proof}

\begin{proposition} \label{FixedPointSetRemainder}
Let  $u$  be a function belonging to the set  $\Xi$ defined in (\ref{FixedPointSet}). Then, given $\delta > 0$,  and $\tau > 0$ satisfying $1 - 20 \tau > 0$,  there is  $C_0$ sufficiently small in (\ref{ScaleFunctionAssumptions}) so that: The term $R^*[u]$ defined in Proposition (\ref{FStarGraphMC}) satisfies the estimate
\begin{align} \notag
\| R^*[u] \|_{0, \alpha} (s, z) & \leq C \zeta^2 \lambda^{2 \epsilon_2} \cosh^{-1} (s) + C \zeta \lambda^\epsilon \lambda^{\epsilon_1}\\ \notag
 & \leq \delta \lambda^{\epsilon_0} (z).
\end{align}
\end{proposition}
\begin{proof}
This is a direct consequence of the definition of $\Xi$ in (\ref{FixedPointSet}), Proposition \ref{FStarGraphMC} and the scale function assumptions in (\ref{ScaleFunctionAssumptions}).
\end{proof}

\begin{theorem} \label{PreciseMainTheorem}
There is $\bar{C}_0  > 0$ and $\zeta > 0$ so that: For $C_0 \in (0, \bar{C}_0)$, The function $\Psi$ defined in (\ref{FixedPointMap}) is continuous and maps  $\Xi$ into $\Xi$ and thus has a  fixed point, which we denote by $u^* \in \Xi$. The surface $F^*[u]$ is then a smooth, minimal embedded surface in the ball
\begin{align}\notag
\{(x, y, z) : x^2 + y^2 + z^2 \leq 1 \}.
\end{align}
\end{theorem}

\begin{definition}
We let $u_1$ be the function obtained by applying Proposition \ref{MainInvertibilityStatement} to the function
\begin{align}\notag
 E (s, z) : = -\dot{\lambda}(z) \tanh(s) + \tilde{L} u_0.
\end{align}
where $u_0$ is the function described in (\ref{PeriodicSolution}).
\end{definition}

\begin{proposition}
Set $v_0 : = u_1 + u_0$. Then the following statements hold:
\begin{enumerate}
\item $v_0$ satisfies the equation:
\begin{align} \notag
\tilde{\mathcal{L}} v_0 = \tilde{H}^*[0]
\end{align}
\item $v_0$ satisfies the estimate:
\begin{align}\notag
\| v_0 :  C^{2,\alpha} (\Lambda, \lambda^{\epsilon_2})\| \leq C
\end{align}
\end{enumerate}
\end{proposition}
\begin{proof} [Proof of Theorem \ref{PreciseMainTheorem}]
Proposition \ref{FStarGraphMC} gives that 
\begin{align} \notag
\tilde{H}^*[u ]  = \tilde{H}^*[0] + \tilde{\mathcal{L}} u + \tilde{R}^*[u].
\end{align}
This then gives that 
\begin{align} \notag
\Psi[u] - \Psi[v_0]  &= u - v_0 - \tilde{\mathcal{L}}^{-1}\left( \tilde{H}^*[u] - \tilde{H}^*[v_0]\right) \\ \notag
& = \mathcal{L}^{-1} \left( \tilde{R}^*[u] - \tilde{R}^*[v_0]\right)
\end{align}
We then have From Proposition \ref{MainInvertibilityStatement} and Proposition \ref{FixedPointSetRemainder}   that 
\begin{align} \notag
\|  \mathcal{L}^{-1} \left( \tilde{R}^*[u] - \tilde{R}^*[v_0]\right) : C^{2, \alpha} (\Lambda, \lambda^{\epsilon_2})\| \leq C(2 \delta)
\end{align}
Taking $\zeta$ sufficiently large and $\delta$ sufficiently small gives that $\Psi(\Xi) \subset (\Xi)$. Since the map $\Psi$ depends continuously on $u$, the Schauder Fixed Point Theroem gives the existence of a function $u^* \in \Xi$ such that 
\begin{align} \notag
\Psi(u^*) = u^*.
\end{align}
It then follows directly that $\tilde{H}^*[u] = 0$, so that $F^*[u]$ is a $C^{2, \alpha}$ minimal immersion. Higher estimates follow from standard elliptic theory.  The embeddedness of $F^*[u]$ in the unit ball around the origin follows directly from the definition of $\Xi$. 
\end{proof}
\bibliographystyle{amsalpha}

\begin{thebibliography}{99}

\bibitem[BT]{BT}
Jacob Bernstein and Giuseppe Tinaglia, \emph{Topological Type of Limit Laminations of Embedded Minimal Disks}, To appear, J. Diff. Geom., arXiv:1309.6260.

\bibitem[BK1]{BK1}
Christine Breiner and Stephen J. Kleene, \emph{A minimal lamination of the interior of a positive cone with quadratic curvature blowup},  To appear, "J. Geom. Anal.",  arXiv:1409.8381.


\bibitem[BK2]{BK2}
Christine Breiner and Stephen J. Kleene, \emph{Logarithmically spiraling helicoids}, submitted preprint,  arXiv:1404.6996.
\bibitem[CM1]{CM1}
Tobias H. Colding and William P. Minicozzi, \emph{ Embedded minimal disks: proper
versus nonproperÑglobal versus local}, Trans. Amer. Math. Soc. 356 (2004), no. 1,
283Ð289 (electronic). MR2020033 (2004k:53011).

\bibitem[CM2]{CM2}
Tobias H. Colding and William P. Minicozzi, \emph{The space of embedded minimal surfaces of fixed genus in a 3-manifold.
II. Multi-valued graphs in disks},  Ann. of Math. (2) 160 (2004), no. 1, 69Ð92, DOI
10.4007/annals.2004.160.69. MR2119718 (2006a:53005).

\bibitem[CM3]{CM3}
Tobias H. Colding and William P. Minicozzi, \emph{The space of embedded minimal surfaces of fixed genus in a 3-manifold. IV.
Locally simply connected}, Ann. of Math. (2) 160 (2004), no. 2, 573Ð615. MR2123933
(2006e:53013).

\bibitem[CM4]{CM4}
Tobias H. Colding and William P. Minicozzi II, \emph{Embedded minimal disks: proper versus nonproper - global versus
local,} Transactions of A.M.S. 356 (2003), no. 1, 283Ð289.

\bibitem[D]{D}
Brian Dean, \emph{Embedded minimal disks with prescribed curvature blowup}, Proc. Amer. Math. Soc. 134 (2006),
no. 4, 1197Ð1204.
\bibitem[E]{Ev}
 Lawrence C. Evans, \emph{Partial differential equations}, AMS, 1997.
 
 \bibitem[HW]{HW}
David Hoffman and Brian White,  \emph{Sequences of embedded minimal disks whose curvatures blow up on a prescribed subset of a line},
Comm. Anal. Geom. 19 (2011), no. 3, 487Ð502.

\bibitem[Kh]{Kh}
Siddique Khan, \emph{A minimal lamination of the unit ball with singularities along a line segment}, Illinois J. Math. 53
(2009), no. 3, 833Ð855 (2010).

\bibitem[K1]{K1}
 Stephen  J. Kleene, \emph{ A minimal lamination with Cantor set-like singularities}, Proc. Amer. Math. Soc. 140 (2012),
no. 4, 1423Ð1436. 

\bibitem[MPR]{MPR}
 William H. Meeks III, Joaquin Perez, and Antonio Ros. Local removable singularity theorems for minimal laminations. Preprint
at http://arxiv.org/abs/1308.6439.

\bibitem[MW]{MW}
William H. Meeks III and Matthias Weber, \emph{Bending the helicoid}, Math. Ann. 339 (2007), no. 4, 783Ð798.
 \bibitem[W1]{W1}
 Brian White, \emph{Curvatures of embedded minimal disks blow up on subsets of C1 curves}
(2011).






\end{thebibliography}

 \end{document}